\documentclass[a4paper,12pt,draft]{amsart}

\usepackage[margin=30mm]{geometry}
\usepackage{amssymb, amsmath, amsthm, amscd}

\usepackage[T1]{fontenc}
\usepackage{eucal,mathrsfs,dsfont}
\usepackage{color}
\usepackage{verbatim}

\newcommand{\rene}{\color{black}}
\newcommand{\normal}{\color{black}}

\definecolor{mno}{rgb}{0.5,0.1,0.5}

\renewcommand{\leq}{\leqslant}
\renewcommand{\geq}{\geqslant}
\renewcommand{\le}{\leqslant}
\renewcommand{\ge}{\geqslant}
\renewcommand{\Re}{\ensuremath{\operatorname{Re}}}

\newcommand{\R}{\mathds R}

\newcommand{\Rd}{{\mathds R^d}}
\newcommand{\Pp}{\mathds P}
\newcommand{\Ee}{\mathds E}
\newcommand{\I}{\mathds 1}
\newcommand{\Z}{\mathds Z}

\newcommand{\var}{\mathrm{Var}}
\newcommand{\Bb}{\mathscr{B}}
\newcommand{\scalp}[2]{\langle#1,#2\rangle}
\newcommand{\lrscalp}[2]{\left\langle#1,#2\right\rangle}
\newcommand{\sfera}{\mathds{S}}

\newcommand{\id}{\operatorname{id}}
\newcommand{\supp}{\operatorname{supp}}
\newcommand{\Rank}{\operatorname{Rank}}
\newcommand{\Leb}{\operatorname{Leb}}

\newcommand{\nnu}{\mathsf{n}}
\newcommand{\mmu}{\mathsf{\pi}}

\newtheorem{theorem}{Theorem}[section]
\newtheorem{lemma}[theorem]{Lemma}
\newtheorem{proposition}[theorem]{Proposition}
\newtheorem{corollary}[theorem]{Corollary}

\theoremstyle{definition}

\newtheorem{example}[theorem]{Example}
\newtheorem{remark}[theorem]{Remark}
\newtheorem*{ack}{Acknowledgement}

\begin{document}

\title[On the Coupling Property of OU Processes]{\bfseries On the Coupling Property and the Liouville Theorem for Ornstein-Uhlenbeck Processes}

\author{Ren\'{e} L.\ Schilling \qquad Jian Wang}
\thanks{\emph{R.\ Schilling:} TU Dresden, Institut f\"{u}r Mathematische Stochastik, 01062 Dresden, Germany. \texttt{rene.schilling@tu-dresden.de}}
\thanks{\emph{J.\ Wang:}
School of Mathematics and Computer Science, Fujian Normal
University, 350007, Fuzhou, P.R. China \emph{and} TU Dresden,
Institut f\"{u}r Mathematische Stochastik, 01062 Dresden, Germany.
\texttt{jianwang@fjnu.edu.cn}}

\date{}

\maketitle

\begin{abstract} Using a coupling for the weighted sum of
independent random variables and the explicit expression of the
transition semigroup of Ornstein-Uhlenbeck processes driven by
compound Poisson processes, we establish the existence of
a successful coupling and the Liouville theorem for general
Ornstein-Uhlenbeck processes. Then we present the explicit coupling
property of Ornstein-Uhlenbeck processes directly from the behaviour
of the corresponding symbol or characteristic exponent. This
approach allows us to derive
gradient estimates for Ornstein-Uhlenbeck processes via the symbol.

\medskip


\medskip

\end{abstract}

\section{Main Results}\label{section1}
Let $(X^x_t)_{t\ge0}$ be an $n$-dimensional Ornstein-Uhlenbeck
process, which is defined as the unique strong solution of the
following stochastic differential equation
\begin{equation}\label{ou1}
    dX_t = AX_t\,dt + B\,dZ_t,\qquad X_0=x\in\R^n.
\end{equation}
Here $A$ is a real $n\times n$ matrix, $B$ is a real $n\times d$
matrix and $Z_t$ is a L\'{e}vy process in $\R^d$; note that we allow $Z_t$ to take values in a proper subspace of $\R^d$. It is well
known that
$$
    X_t^x
    =e^{tA}x + \int_0^t e^{(t-s)A}B\,dZ_s.
$$
The characteristic exponent or symbol $\Phi$ of $Z_t$, defined by
$$
    \Ee\bigl(e^{i\scalp{\xi}{Z_t}}\bigr)
    =e^{-t\Phi(\xi)},\quad \xi\in\R^d,
$$
enjoys the following L\'{e}vy-Khintchine representation:
\begin{equation}\label{ou2}
     \Phi(\xi)
     =\frac{1}{2}\scalp{Q\xi}{\xi} +i\scalp{b}{\xi} +\int_{z\neq 0} \Bigl(1-e^{i\scalp{\xi}{z}}+i\scalp{\xi}{z}\I_{B(0,1)}(z)\Bigr)\nu(dz),
\end{equation}
where $Q=(q_{j,k})_{j,k=1}^d$ is a positive semi-definite matrix,
$b\in\Rd$ is the drift vector and $\nu$ is the L\'evy measure, i.e.\
a $\sigma$-finite measure on $\R^d\setminus\{0\}$ such that
$\int_{z\neq 0}(1\wedge |z|^2)\,\nu(dz)<\infty$. For every
$\varepsilon>0$, define ${\nu}_\varepsilon$ on $\R^d$ as follows:
\begin{equation*}
    {\nu}_\varepsilon(C)
    =
    \begin{cases}
        \nu(C),                           & \text{if\ \ } \nu(\R^d)<\infty;\\
        \nu(C\setminus \{z: |z|<\varepsilon\}), & \text{if\ \ } \nu(\R^d)=\infty.
    \end{cases}
\end{equation*}

Let $(Y_t)_{t\ge0}$ be a Markov process on $\R^n$ with transition function $P_t(x,\cdot)$. Then, according to \cite{Li,T, SW}, we say that $(Y_t)_{t\ge0}$
admits a \emph{successful coupling} (also: enjoys the \emph{coupling property}) if for any $x,y\in\R^n$,
\begin{equation*}\label{prex1}
    \lim_{t\rightarrow\infty}\|P_t(x,\cdot)-P_t(y,\cdot)\|_{\var}=0,
\end{equation*}
where $\|\cdot\|_{\var}$ stands for the total variation
norm. If a Markov process admits a successful coupling, then it also
has the Liouville property, i.e.\ every bounded harmonic function is
constant; in this context a function $f$ is harmonic, if $Lf=0$
where $L$ is the generator of the Markov process. See \cite{CG,CW}
and the references therein for this result and more details on the
coupling property.

Let $A$ be an $n\times n$ matrix. We say that an eigenvalue $\lambda$ of $A$ is
\emph{semisimple} if the dimension of the corresponding eigenspace
is equal to the algebraic multiplicity of $\lambda$ as a root of characteristic polynomial of
$A$. Note that for symmetric matrices $A$ all eigenvalues are real and semisimple.
Recall that for any two bounded measures $\mu$ and $\nu$ on $(\R^d,\Bb(\R^d))$,
$\mu\wedge\nu:=\mu-(\mu-\nu)^+$, where $(\mu-\nu)^{\pm}$ refers to the
Jordan-Hahn decomposition of the signed measure $\mu-\nu$. In
particular, $\mu\wedge\nu=\nu\wedge\mu$, and $\mu\wedge
\nu\,(\R^d)=\frac{1}{2}\big[\mu(\R^d)+\nu(\R^d)-\|\mu-\nu\|_{\var}\big].$

\medskip

One of our main results is the following
\begin{theorem}\label{th1}
    Let $P_t(x,\cdot)$ be the transition probability of the Ornstein-Uhlenbeck process $\{X_t^x\}_{t\ge0}$ given by \eqref{ou1}. Assume that $\Rank(B)=n$ (which implies $n\le d$), and that there exist $\varepsilon,\delta>0$ such that
\begin{equation}\label{th2233}
    \inf_{z\in\R^d,|z|\le \delta}\nu _\varepsilon\wedge (\delta_z*\nu_\varepsilon)(\R^d)>0.
\end{equation}

      If the real parts of all eigenvalues of $A$ are non-positive and if all purely imaginary eigenvalues are semisimple, then there exists a constant $C=C(\varepsilon,\delta,\nu,A,B)>0 $ such that for all $x,y\in\R^n$ and $t>0$,
\begin{equation}\label{th21}
    \|P_t(x,\cdot)-P_t(y,\cdot)\|_{\var}
    \le
    \frac{C(1+|x-y|)}{\sqrt{t}}\wedge2.
\end{equation}


\end{theorem}

As a consequence of Theorem \ref{th1}, we immediately obtain the following result which partly answers the following question about Liouville theorems for non-local operators from \cite[page 458]{PZ}: \emph{A challenging task would be to apply other probabilistic techniques, based on ... coupling to non-local operators}.

\begin{corollary}
    Under the conditions of Theorem \ref{th1}, the Ornstein-Uhlenbeck process $\{X_t^x\}_{t\ge0}$ admits a successful coupling and has the Liouville property.
\end{corollary}

\begin{remark}[The conditions of Theorem \ref{th1} are optimal]\label{remarkth1}
(1) If $A=0$, $d=n$ and $B=\id_{\R^n}$, then $X_t$ is just a L\'{e}vy process on
$\R^n$. The condition \eqref{th2233} is one possibility to guarantee
sufficient jump activity such that the L\'{e}vy process $X_t$ admits
a successful coupling. To see that \eqref{th2233} is sharp, we can
use the example in \cite[Remark 1.2]{SW}.

(2) \rene Let $Z_t$ be a (rotationally symmetric) $\alpha$-stable L\'{e}vy process
$Z_t$, $0<\alpha<2$, and denote by $X_t$ the $n$-dimensional Ornstein-Uhlenbeck process driven
by $Z_t$, \normal i.e.\
$$
    dX_t = AX_t\,dt + dZ_t.
$$
If at least one eigenvalue of $A$ has positive real part, then $X_t$ does not have the coupling
property. Indeed, according to \rene \cite[Example 3.4 and Theorem 3.5]{PZ}, \normal we know that $X_t$ does not have the Liouville property, i.e.\ there exists a bounded harmonic function which is not constant. According to \cite[Theorem 21.12]{Li} or \cite[Theorem 1 and its second remark]{CG}, $X_t$ does not have the coupling property. This example indicates that the non-positivity of the real parts of the eigenvalues of $A$ is also necessary.
\end{remark}

\begin{remark}[Strong Feller property vs.\ coupling property]
    In \cite[Theorem 4.1 and Corollary 4.2]{SW} we show that L\'{e}vy processes which have the strong Feller property admit the coupling property. A similar conclusion, however, does not hold for general Ornstein-Uhlenbeck processes. Consider, for instance, the one-dimensional Ornstein-Uhlenbeck process given by
    $$
        dX_t=X_t\,dt+dZ_t,\qquad X_0=x\in\R,
    $$
    where $Z_t$ is an $\alpha$-stable L\'{e}vy process $Z_t$ on $\R$. According to \cite[Theorem 1.1]{PZ2} (or \cite[Theorem A]{NS}) and \cite[Proposition 2.1]{PZ2}, we know that $X_t$ has the strong Feller property. However, the argument used in Remark \ref{remarkth1} shows that this process fails to have the coupling property.
\end{remark}

Recently, F.-Y.\ Wang \cite{wang1} has studied the coupling property
of an Ornstein-Uhlenbeck process $X_t$ defined by \eqref{ou1}.
Assume that $\Rank(B)=n$ and $\scalp{Ax}{x}\le 0$ holds for
$x\in\R^n$. In \cite[Theorem 3.1]{wang1} it is proved that \eqref
{th21} is satisfied for some constant $C>0$, whenever the L\'{e}vy
measure of $Z_t$ satisfies $\nu(dz)\ge \rho_0(z)dz$ such that
\begin{equation}\label{wang22}
    \int_{\{|z-z_0|\le \varepsilon\}}
    \frac{dz}{\rho_0(z)}<\infty
\end{equation}
holds for some $z_0\in\R^d$ and some $\varepsilon>0$.

Let us compare F.-Y.\ Wang's result with our Theorem \ref{th1}.

\begin{proposition}\label{improvement}
    Assume that \eqref{wang22} holds for some $\rho_0\in L^1_{\mathrm{loc}}(\R^d\setminus\{0\})$, some $z_0\in\R^d$ and some $\varepsilon>0$. Then, there exist a closed subset $F\subset \overline{B(z_0,\varepsilon)}=\{z\in\R^d: |z-z_0|\le \varepsilon\}$ and a constant $\delta>0$ such that
    $$
        \inf_{x\in\R^d, |x|\le \delta}\int_F\big(\rho_0(z)\wedge\rho_0(z-x)\big)\,dz>0.
    $$
\end{proposition}
    We postpone the technical proof of Proposition \ref{improvement} to Section \ref{subsec-appendix2} in the appendix. Proposition \ref{improvement} shows
    that Theorem \ref{th1} improves \cite[Theorem 3.1]{wang1}, even if the L\'{e}vy measure $\nu$ of $Z_t$ has an absolutely continuous component as we will see in the following example.

\begin{example}
    Let $C_{3/4}$ be a Smith-Volterra-Cantor set in $[0,1]$ with Lebesgue measure $\Leb(C_{3/4})=3/4$, i.e.\ $C_{3/4}$ is a perfect set with empty interior, see e.g.\ \cite[Chapter 3, Section 18]{AB}. Consider the following one-dimensional Ornstein-Uhlenbeck process
    $$
        dX_t=-X_t\,dt + dZ_t,\qquad X_0=x\in\R,
    $$
    where $Z_t$ is a real-valued L\'{e}vy process with L\'{e}vy measure $\nu(dz)=\I_{C_{3/4}}(z)\,dz$. We will see that we can use Theorem \ref{th1} to show the coupling property of the process $X_t$ while the criterion from \cite[Theorem 3.1]{wang1} fails.

    Let $\delta \in (0, 1/8)$ and $z\in [-\delta, \delta]$. Then
    \begin{align*}
        \nu _\varepsilon\wedge (\delta_z*\nu_\varepsilon)(\R)
        &=\int\Bigl(\I_{C_{3/4}}(x)\wedge\I_{C_{3/4}}(x+z)\Bigr)dx\\
        &=\Leb\bigl( C_{3/4}\cap (C_{3/4}-z)\bigr)\\
        &=\Leb(C_{3/4}) + \Leb(C_{3/4}-z) - \Leb\bigl( C_{3/4}\cup (C_{3/4}-z)\bigr)\\
        &\geq \frac 64 - \Leb[-|z|,1+|z|]\geq\frac 14.
    \end{align*}
    This shows that the conditions of Theorem \ref{th1} are satisfied.

    On the other hand, since $C_{3/4}$ contains no intervals, we see that for all $z_0\in\R$ and $\varepsilon>0$,
    $$
        \int_{\{|z-z_0|\le \varepsilon\}}\frac{dz}{\I_{C_{3/4}}(z)}=\infty
    $$
    (here we use the convention $\frac 10 = +\infty$). This means that \eqref{wang22} does not hold.
\end{example}

\medskip

Now we are going to estimate $\|P_t(x,\cdot)-P_t(y,\cdot)\|_{\var}$ for large values of $t$ with the help of the characteristic exponent $\Phi(\xi)$ of the L\'evy process $Z_t$.
 We restrict ourselves to the case where $Q=0$ in \eqref{ou2}, i.e.\ to L\'{e}vy process $(Z_t)_{t\ge 0}$ without a Gaussian part. For $t, \rho>0$, define
$$
    \varphi_t(\rho)
    :=\sup_{|\xi|\le \rho}\int_0^t\Re\Phi\big(B^\top e^{sA^\top}\xi\big)\,ds,
$$
where $M^\top$ denotes the transpose of the matrix $M$.
\begin{theorem}\label{coup}
    Let $P_t(x,\cdot)$ be the transition function of the Ornstein-Uhlenbeck process $\{X_t^x\}_{t\ge0}$ on $\R^n$ given by \eqref{ou1}. Assume that there exists some $t_0 > 0$ such that
    \begin{equation}\label{coup1}
       \liminf\limits_{|\xi|\rightarrow\infty}\frac{\int_0^{t_0}\Re\Phi\big(B^\top e^{sA^\top}\xi\big)\,ds}{\log (1+|\xi|)} >2n+2.
\end{equation}
If
\begin{equation}\label{coup2}
    \int \exp\left(-\int_0^t\Re\Phi\big(B^\top e^{sA^\top}\xi\big)ds\right) |\xi|^{n+2} \,d\xi
    = \mathsf{O}\left(\varphi_t^{-1}(1)^{2n+2}\right)
    \quad\text{as\ \ } t\to\infty,
\end{equation}
    then there exist $t_1,C>0$ such that for any $x,y\in\R^n$ and $t\ge t_1$,
\begin{equation}\label{coup3}
    \|P_t(x,\cdot)-P_t(y,\cdot)\|_{\var}
    \le C|e^{tA}(x-y)|\,\varphi^{-1}_t(1).
\end{equation}
In particular, when \begin{equation}\label{coup4}\xi\mapsto\int_0^\infty\Re\Phi\big(B^\top
e^{sA^\top}\xi\big)\,ds \quad\textrm{ is locally bounded},\end{equation} we only
need the condition \eqref{coup1} to get \eqref{coup3}.
\end{theorem}

Note that \eqref{coup4} is, e.g.\ satisfied, if the real parts of
all eigenvalues of $A$ are negative and
$$\limsup_{|\xi|\rightarrow0}\frac{\Re\Phi\big(B^\top
\xi\big)}{|\xi|^\kappa}<\infty$$ for some constant $\kappa>0$.

\bigskip

The remaining part of this paper is organized as follows. In Section \ref{section2} we first present the proof of Theorem \ref{th1}, where a coupling for the weighted sum of independent random variables and the explicit expression of the transition semigroup of Ornstein-Uhlenbeck processes driven by a compound Poisson process are used. Then, we follow the approach of our recent paper \cite{SSW} to prove Theorem \ref{coup}. As a byproduct, we also derive explicit gradient estimates for Ornstein-Uhlenbeck processes, cf.\ the Appendix \ref{subsec-appendix1}.

\section{Proofs of Theorems}\label{section2}
We begin with the proof of Theorem \ref{th1}.

\begin{proof}[Proof of Theorem \ref{th1}]
The proof is split into six steps. 

\medskip
\noindent
\emph{Step 1}.
For any $\varepsilon>0$, let $(Z_t^\varepsilon)_{t\ge0}$ be a compound Poisson process on $\R^d$ whose L\'{e}vy measure is $\nu_\varepsilon$. Then, $(Z_t^\varepsilon)_{t\ge0}$ and $(Z_t-Z_t^\varepsilon)_{t\ge0}$ are independent L\'{e}vy processes. It follows, in particular, that the random variables
$$
    X_t^{\varepsilon,x}:=e^{tA}x+\int_0^t e^{(t-s)A}B\,dZ_s^\varepsilon
$$
and
$$
    X_t^x-X_t^{\varepsilon,x}:=\int_0^t e^{(t-s)A}B\,d(Z_s-Z_s^\varepsilon)
$$
are independent for any $\varepsilon>0$ and $t\ge0$.

\medskip
\noindent
\emph{Step 2}.
Denote by $\mu_{\varepsilon,t}$ the law of random variable
$$
    X_t^{\varepsilon,0} := X_t^{\varepsilon,x}-e^{tA}x = \int_0^te^{(t-s)A}B\,dZ_s^\varepsilon.
$$
We will compute $\mu_{\varepsilon,t}$, which coincides with the law
of $\int_0^te^{sA}B\,dZ_s^\varepsilon$, cf.\ Lemma
\ref{lemmaproofth11} below. Our argument follows the proof of
\cite[Theorem 1.1]{PZ2}, which is motivated by \cite[Theorem
27.7]{SA}.

The law of the compound poisson process $Z_t^\varepsilon$ is given by
$$
    e^{-C_\varepsilon t}\bigg[\delta_0+\sum_{k=1}^\infty \frac{(C_\varepsilon t)^k}{k!}\,\bar{\nu}_\varepsilon^{*k}\bigg],
$$
where $C_\varepsilon=\nu_\varepsilon(\R^d)$, $\bar{\nu}_\varepsilon=\nu_\varepsilon/C_\varepsilon$ and
$\bar{\nu}_\varepsilon^{*k}$ is the $k$-fold convolution of $\bar{\nu}_\varepsilon$.

Construct a sequence $(\xi_i)_{i\ge1}$ of iid random variables which are exponentially distributed with intensity $C_\varepsilon$, and introduce a further sequence $(U_i)_{i\ge1}$ of iid random variables on $\R^d$ with law $\bar{\nu}_\varepsilon$. We will assume that the random variables $(U_i)_{i\ge1}$ are independent of the sequence $(\xi_i)_{i\ge1}$. It is not difficult to check that the random variable
\begin{equation}\label{proofs0}
    0\cdot\I_{\{\xi_1>t\}} +\sum_{k=1}^\infty\I_{\{\xi_1+\cdots+\xi_k\le t<\xi_1+\cdots+\xi_{k+1}\}} \Big(e^{\xi_1A}BU_1+\cdots+e^{(\xi_1+\cdots+\xi_k)A}BU_k\Big)
\end{equation}
also has the probability distribution $\mu_{\varepsilon,t}$.

Using \eqref{proofs0} we find for any $f\in B_b(\R^n)$,
\begin{equation}\label{proofs1}
    \Ee f\bigl(X_t^{\varepsilon,x}\bigr)
    =\int f\bigl(e^{tA }x+z\bigr)\,\mu_{\varepsilon,t}(dz)
    =f\bigl(e^{tA}x\bigr)\,e^{-C_\varepsilon t}+Hf(x),
\end{equation}
where
\begin{align*}
&Hf(x)\\
&:=\Ee f\left(\sum_{k=1}^\infty\I_{\{\xi_1+\cdots+\xi_k\le t<\xi_1+\cdots+\xi_{k+1}\}} \Bigl(e^{tA}x+e^{\xi_1A}BU_1+\cdots+e^{(\xi_1+\cdots+\xi_k)A}BU_k\Bigr)\right)\\
&=\sum_{k=1}^\infty\Ee f\left(\I_{\{\xi_1+\cdots+\xi_k\le t<\xi_1+\cdots+\xi_{k+1}\}} \Bigl(e^{tA}x+e^{\xi_1A}BU_1+\cdots+e^{(\xi_1+\cdots+\xi_k)A}BU_k\Bigr)\right)\\
&=\sum_{k=1}^\infty \mathop{\int\cdots\int}\limits_{t_1+\cdots+t_k\le t<t_1+\cdots+t_{k+1}} C_\varepsilon^{k+1}e^{-C_\varepsilon(t_1+\cdots+t_{k+1})}\,dt_1\cdots dt_{k+1}{\;\;\times}\\
&\qquad\quad{\times}\;\int\limits_{\R^d}\cdots\int\limits_{\R^d} f\bigl(e^{tA}x+e^{t_1A}By_1+\cdots+e^{(t_1+\cdots+t_k)A}By_k\bigr) \,\bar{\nu}_\varepsilon(dy_1)\cdots \bar{\nu}_\varepsilon(dy_k)\\
&=\sum_{k=1}^\infty \mathop{\int\cdots\int}\limits_{t_1+\cdots+t_k\le t<t_1+\cdots+t_{k+1}} C_\varepsilon^{k+1}e^{-C_\varepsilon(t_1+\cdots+t_{k+1})}\,dt_1\cdots dt_{k+1}\;\;\times\\
&\qquad\quad\times\;\int\limits_{\R^n} f\bigl(e^{tA}x+z\bigr)\,\mu_{t_1,\cdots,t_k}(dz).
\end{align*}
Here $\mu_{t_1,\cdots,t_k}$ is the probability measure on $\R^n$ which is the image of the $k$-fold product measure
$\bar{\nu}_\varepsilon\times \cdots\times \bar{\nu}_\varepsilon$ under the linear transformation $J_{t_1,\ldots,t_k}$
(independent of $\varepsilon$) acting from $(\R^{d})^k$ into $\R^n$:
$$
    J_{t_1,\ldots,t_k}(y_1,\ldots,y_k)
    =e^{t_1A}B y_1 + \cdots +e^{(t_1+\cdots+t_k)A}B y_k,
$$
for $y_i\in\R^d$ and $i=1,\cdots,k$.

\medskip
\noindent \emph{Step 3}. Let $P_t(x,\cdot)$ and $P_t$ be the
transition function and the transition semigroup of the
Ornstein-Uhlenbeck process $(X^x_t)_{t\ge0}$. Similarly, we denote
by  $P^\varepsilon_t(x,\cdot)$ and $P^\varepsilon_t$ the transition
function and the transition semigroup of
$(X_t^{\varepsilon,x})_{t\ge0}$, and by $Q^\varepsilon_t(x,\cdot)$
and $Q^\varepsilon_t$ the transition function and the transition
semigroup of $(X_t^x-X_t^{\varepsilon,x})_{t\ge0}$. By the
independence of the processes $(X_t^{\varepsilon,x})_{t\ge0}$ and
$(X_t^x-X_t^{\varepsilon,x})_{t\ge0}$, we get
\begin{equation}\label{proofs2}\begin{aligned}
    \|P_t(x,\cdot)-P_t(y,\cdot)\|_{\var}
    &= \sup_{\|f\|_\infty\le 1}\big|P_tf(x)-P_tf(y)\big|\\
    &=\sup_{\|f\|_\infty\le 1} \big|P_t^\varepsilon Q_t^\varepsilon f(x)-P_t^\varepsilon Q_t^\varepsilon f(y)\big|\\
    &\le \sup_{\|h\|_\infty\le 1} \big|P^\varepsilon_th(x)-P^\varepsilon_th(y)\big|.
\end{aligned}\end{equation}
Furthermore, it follows from \eqref{proofs1} that
\begin{equation}\label{proofs3}\begin{aligned}
   &\sup_{\|h\|_\infty\le 1} \big|P^\varepsilon_th(x)-P^\varepsilon_th(y)\big|\\
    &\le 2e^{-C_\varepsilon t}+\sum_{k=1}^\infty \mathop{\int\cdots\int}\limits_{t_1+\cdots+t_k\le t<t_1+\cdots+t_{k+1}} C_\varepsilon^{k+1}e^{-C_\varepsilon(t_1+\cdots+t_{k+1})}\,dt_1\cdots dt_{k+1}\;\;\times\\
    &\qquad\times\sup_{\|h\|_\infty\le 1} \bigg|\int\limits_{\R^n} h\big(e^{tA}x+z\big)\,\mu_{t_1,\cdots,t_k}(dz) - \int\limits_{\R^n} h\big(e^{tA}y+z\big)\,\mu_{t_1,\cdots,t_k}(dz)\biggr|\\
    &= 2e^{-C_\varepsilon t}+\sum_{k=1}^\infty \mathop{\int\cdots\int}\limits_{t_1+\cdots+t_k\le t<t_1+\cdots+t_{k+1}} C_\varepsilon^{k+1}e^{-C_\varepsilon(t_1+\cdots+t_{k+1})}\,dt_1\cdots dt_{k+1}\;\;\times\\
    &\qquad\times\sup_{\|h\|_\infty\le 1} \bigg|\int\limits_{\R^n} h\big(e^{tA}(x-y)+z\big)\,\mu_{t_1,\cdots,t_k}(dz) - \int\limits_{\R^n} h(z)\,\mu_{t_1,\cdots,t_k}(dz)\bigg|\\
    &\le 2e^{-C_\varepsilon t}+\sum_{k=1}^\infty\mathop{\int\cdots\int}\limits_{t_1+\cdots+t_k\le t<t_1+\cdots+t_{k+1}} C_\varepsilon^{k+1}e^{-C_\varepsilon(t_1+\cdots+t_{k+1})}\,dt_1\cdots dt_{k+1}\;\;\times\\
&\qquad\times\|\delta_{e^{tA}(x-y)}*\mu_{t_1,\cdots,t_k}-\mu_{t_1,\cdots,t_k}\|_{\var}.
\end{aligned}\end{equation}

\medskip
\noindent
\emph{Step 4}.
For any $a\in\R^n$, $a\neq 0$, let $R_a$ be the non-degenerate rotation such that $R_a a=|a|e_1$. Then, by \cite[Lemma 3.2]{SW},
\begin{align*}
    \big\|\delta_{e^{tA}(x-y)}*\mu_{t_1,\cdots,t_k} &-\mu_{t_1,\cdots,t_k}\big\|_{\var}\\
    &=\big\|\delta_{|e^{tA}(x-y)|e_1}*\big(\mu_{t_1,\cdots,t_k}\circ R_{e^{tA}(x-y)}^{-1}\big)-\mu_{t_1,\cdots,t_k}\circ R_{e^{tA}(x-y)}^{-1}\big\|_{\var}.
\end{align*}
Since
$\mu_{t_1,\cdots,t_k}$ is the law of the random variable
$$
    \sum_{i=1}^ke^{(t_1+\cdots+t_i)A}BU_i,
$$
$\mu_{t_1,\cdots,t_k}\circ R_{e^{tA}(x-y)}^{-1}$ is the law of the random variable
$$
    \sum_{i=1}^kR_{e^{tA}(x-y)}\big(e^{(t_1+\cdots+t_i)A}BU_i\big).
$$

To estimate $\big\|\delta_{|e^{tA}(x-y)|e_1}*\big(\mu_{t_1,\cdots,t_k}\circ R_{e^{tA}(x-y)}^{-1}\big)-\mu_{t_1,\cdots,t_k}\circ R_{e^{tA}(x-y)}^{-1}\big\|_{\var}$, we will use the Mineka and Lindvall-Rogers couplings for random walks. The remainder of this part is based on the proof of \cite[Proposition 3.3]{SW}. In order to ease notations, we set $\nnu:=\bar{\nu}_\varepsilon$ and $\nnu^{a}:=\delta_a*\bar\nu_\varepsilon$ for any $a\in\R^d$.

Since $\Rank(B)=n$, there exists a real $d\times n$ matrix $\bar{B}$ such that $B\bar{B}=\id_{\R^n}$, see e.g.\ \cite[Theorem 2.6.1, Page 35]{Bern}.  For any $i\ge 1$, let $(U_i,\Delta U_i)\in \R^d \times \R^d$ be a pair of random variables with the following distribution
$$
    \Pp\big((U_i,\Delta U_i)\in C\times D\big)
    =
    \begin{cases}
    \qquad \frac 12 (\nnu\wedge\nnu^{-a_i})(C), & \text{if\ \ } D=\{a_i\};\\
    \qquad \frac 12 (\nnu\wedge\nnu^{a_i})(C),  & \text{if\ \ } D= \{-a_i\};\\
    \big(\nnu- \frac 12 (\nnu\wedge\nnu^{-a_i}+\nnu\wedge\nnu^{a_i})\big)(C), & \text{if\ \ } D=\{0\};
    \end{cases}
$$
where $C\in\Bb(\R^d)$, $a_i=\bar B\, e^{(t-(t_1+\cdots+t_i))A}\,(x-y)$ and $D$ is any of the following three sets: $\{-a_i\}$, $\{0\}$ or $\{a_i\}$.
Again by \cite[Lemma 3.2]{SW},
\begin{align*}
    \Pp\big(\Delta U_i=-a_i\big)
    &=\frac{1}{2}\big(\nnu\wedge\big(\delta_{a_i}*\nnu)\big)(\R^d)\\
    &=\frac{1}{2}\big(\nnu\wedge\big(\delta_{-a_i}*\nnu)\big)(\R^d)\\
    &=\Pp(\Delta U_i=a_i).
\end{align*}
It is clear that the distribution of $U_i$ is $\nnu$. Let $U_i'=U_i+\Delta U_i$. We claim that the distribution of $U_i'$ is
also $\nnu$. Indeed, for any $C\in\mathscr{B}(\R^d)$,
\begin{align*}
    &\Pp(U_i'\in C)\\
    &=\Pp(U_i-a_i\in C, \Delta U_i=-a_i)
        + \Pp(U_i+a_i\in C, \Delta U_i=a_i)
        +\Pp(U_i \in A, \Delta U_i=0)\\
    &= \frac 12\left(\delta_{-a_i}*(\nnu\wedge\nnu^{a_i})\right)(C)
        +\! \frac12\left(\delta_{a_i}*(\nnu\wedge\nnu^{-a_i})\right)(C)
       \! +\! \left(\!\!\nnu-\!\! \frac 12\,\big(\nnu\wedge\nnu^{-a_i}+\nnu\wedge\nnu^{a_i}\big)\!\!\right)(C)\\
    &=\nnu(C),
\end{align*}
where we have used that
$$
    \delta_{a_i}*(\nnu\wedge\nnu^{-a_i})=\nnu\wedge\nnu^{a_i}\quad\textrm{ and }\quad\delta_{-a_i}*(\nnu\wedge\nnu^{a_i})=\nnu\wedge\nnu^{- a_i}.
$$

Without loss of generality, we can assume that the pairs $(U_i,U_i')$ are independent for all $i\ge1$.  Now we construct the
coupling
$$
    (S_k,S_k')_{k\ge1}
    =\left(\sum_{i=1}^kR_{e^{tA}(x-y)}\big(e^{(t_1+\cdots+t_i)A}BU_i\big),
        \sum_{i=1}^kR_{e^{tA}(x-y)}\big(e^{(t_1+\cdots+t_i)A}BU'_i\big)\right)_{k\ge1}
$$
of
$$
    S_k:=\sum_{i=1}^kR_{e^{tA}(x-y)}\big(e^{(t_1+\cdots+t_i)A}BU_i\big).
$$
Since $U'_i-U_i=\Delta U_i$ is either $\pm a_i$ or $0$, we know that
\begin{align*}
(S_k-& S_k')_{k\ge1}\\
    &=\left(\sum_{i=1}^kR_{e^{tA}(x-y)}\big(e^{(t_1+\cdots+t_i)A}BU'_i\big) -\sum_{i=1}^kR_{e^{tA}(x-y)}\big(e^{(t_1+\cdots+t_i)A}BU_i\big)\right)_{k\ge1}\\
    &=\left(\sum_{i=1}^kR_{e^{tA}(x-y)}\big(e^{(t_1+\cdots+t_i)A}B(U'_i-U_i)\big)\right)_{k\ge1}
\end{align*}
is a random walk on $\R^n$ whose steps are symmetrically (but not necessarily identically) distributed and take only the values $\pm |e^{tA}(x-y)| e_1$ and $0$.

Set $S^j_{k}=\sum_{i=1}^k\eta^j_{i}$ and $S^{j\,\prime}_{k}=\sum_{i=1}^k\eta^{j\,\prime}_{i}$ for $1\le j\le
n$, where
 $$
    (\eta^{1}_{i},\ldots,\eta^{n}_{i})=R_{e^{tA}(x-y)}\big(e^{(t_1+\cdots+t_i)A}BU_i\big)
    $$ and $$(\eta_{i}^{1\,\prime},\ldots,\eta^{n\,\prime}_{i})=R_{e^{tA}(x-y)}\big(e^{(t_1+\cdots+t_i)A}BU_i^{\prime}\big).
$$
  Then $(S^1_k-S^{1\,\prime}_k)_{k\geq 1}$ is a random walk on $\R$ whose steps are independent and attain the values $-|e^{tA}(x-y)|$, $0$ and $|e^{tA}(x-y)|$ with probabilities $\frac 12(1-p_i)$, $p_i$ and $\frac 12(1-p_i)$, respectively; the values of the $p_i$ are given by
\begin{align*}
    p_i
    :&=\Pp(\eta^{1\,\prime}_i-\eta^1_{i}=0)\\
    &= \left(\nnu- \tfrac 12 (\nnu\wedge\nnu^{-a_i}+\nnu\wedge\nnu^{a_i})\right)(\R^d)\\
    &= 1-\nnu\wedge\nnu^{-a_i}(\R^d).
\end{align*}
Since $S^j_{k}=S^{j\,\prime}_{k}$ for $2\le j\le n$, we get
\begin{equation}\label{proofs6}
  \|\delta_{e^{tA}(x-y)}*\mu_{t_1,\cdots,t_k}-\mu_{t_1,\cdots,t_k}\|_{\var}
    \le 2\,\Pp(T^S>k),
\end{equation}
where
$$
    T^S=\inf\{i\ge1\::\: S^1_{i}=S^{1\,\prime}_{i}+|e^{tA}(x-y)|\}.
$$

\medskip
\noindent \emph{Step 5}.
Since the real parts of all eigenvalues of $A$ are non-positive and since all purely imaginary eigenvalues are semisimple, we know from \cite[Proposition 11.7.2, Page 438]{Bern} that $C_A:=\sup_{t\ge0}\|e^{tA}\|<\infty$. In particular,
when $t\ge t_1+\cdots+t_i$,
$$
    \big|e^{(t-(t_1+\cdots+t_i))A}(x-y)\big|\le C_A|x-y|.
$$
From \eqref{th2233} we get that for all $i\ge1$ and $x, y\in\R^n$ with $|x-y|\le \delta(C_A\|\bar{B}\|)^{-1}$,
\begin{equation}\label{proofcon}\begin{aligned}
    \frac 12(1-p_i)
    &= \frac{1}{2}\big(\nnu\wedge\big(\delta_{-a_i}*\nnu)\big)(\R^d)\\
    &\ge \frac{1}{2}\inf_{|a|\le C_A\|\bar{B}\||x-y|}\nnu\wedge (\delta_a*\nnu)(\R^d)\\
    &\ge \frac{1}{2}\inf_{|a|\le \delta}\nnu\wedge (\delta_a*\nnu)(\R^d)\\
    &=:\frac{1}{2}\,\gamma(\delta)>0.
\end{aligned}\end{equation}

We will now estimate $\Pp(T^S>k)$. Let $V_i$, $i\geq 1$, be independent symmetric random variables on $\R$, whose distributions are given by
$$
    \Pp(V_i=z)
    = \begin{cases}
        \frac 12(1-p_i),     &\text{if\ \ } z=-|e^{tA}(x-y)|;\\
        \frac 12(1-p_i),     &\text{if\ \ } z=|e^{tA}(x-y)|;\\
        \qquad p_i,          &\text{if\ \ } z=0.
    \end{cases}
$$
Set $Z_k:=\sum_{i=1}^k V_i$. We have seen earlier that
$$
    T^S=\inf\{k\ge 1\::\: Z_k=|e^{tA}(x-y)|\}.
$$

For any $k\ge1$, let
$$
    \eta=\eta(k):=\#\big\{i\::\: i\le k\textrm{ and }V_i\neq 0\big\}
$$
and set $\tilde{Z}_k :=\sum_{i=1}^k\tilde{V}_i$, where $\tilde{V}_i$ denotes the $i$th $V_j$ such that $V_j\neq 0$. Then, $\tilde{Z}_k$ is a symmetric random walk with iid steps which are either $-|e^{tA}(x-y)|$ or $|e^{tA}(x-y)|$ with probability $1/2$. Define
$$
    T^{\tilde{Z}}:=\inf\{k\ge 1\::\: \tilde{Z}_k=|e^{tA}(x-y)|\}.
$$
By \eqref{proofcon},
\begin{equation}\label{lll1}\begin{aligned}
    \Pp(T^S>k)
    &=\Pp\left(T^S>k,\; \eta\ge \frac12\,\gamma(\delta)k\right)
        +\Pp\left(T^S>k,\, \eta\le \frac12\, \gamma(\delta)k\right)\\
    &\le\Pp\left(T^{\tilde{Z}}> \frac12\,\gamma(\delta)k\right)
        +\Pp\bigg(\eta\le \frac{1}{2}\sum_{i=1}^k(1-p_i)\bigg)\\
    &\le\Pp\left(T^{\tilde{Z}}>\frac12\,\gamma(\delta)k\right)
        +\Pp\bigg(\Big|\eta-\sum_{i=1}^k(1-p_i)\Big|\ge\frac{1}{2}\sum_{i=1}^k(1-p_i)\bigg).
\end{aligned}\end{equation}

Note that
$$
    \eta=\eta(k)=\sum_{i=1}^k\zeta_i,
$$
where $\zeta_i = \I_{\{V_i\neq 0\}}$, ${1\le i\le k}$, are independent random variables with
$\Pp(\zeta_i = 0) = p_i$ and $\Pp (\zeta_i=1)=1-p_i$. Chebyshev's inequality shows that
\begin{equation}\label{lll2}\begin{aligned}
    \Pp\bigg(\Big|\eta-\sum_{i=1}^k(1-p_i)\Big|\ge\frac{1}{2}\sum_{i=1}^k(1-p_i)\bigg)
    &\le \frac{4\var(\eta)}{\Big(\sum_{i=1}^k(1-p_i)\Big)^2}\\
    &=\frac{4\sum_{i=1}^kp_i(1-p_i)}{\Big(\sum_{i=1}^k(1-p_i)\Big)^2}\\
    &\le \frac{4(1-\gamma(\delta))\sum_{i=1}^k(1-p_i)}{\Big(\sum_{i=1}^k(1-p_i)\Big)^2}\\
    &\le\frac{4(1-\gamma(\delta))}{\gamma(\delta)k}.
\end{aligned}\end{equation}
For the second and the last inequality we have used \eqref{proofcon}.

On the other hand, by Lemma \ref{lemmaproofth13} below,
\begin{align*}
    \Pp\bigg(T^{\tilde{Z}}>\frac{\gamma(\delta)k}{2}\bigg)
    &=\Pp\bigg(\max_{i\le \big[\frac{\gamma(\delta)k}{2}\big]}\tilde{Z}_i<  |e^{tA}(x-y)|\bigg)\\
    &\leq 2\,\Pp\left(0\le \tilde{Z}_{\big[\frac{\gamma(\delta)k}{2}\big]} \leq |e^{tA}(x-y)|\right).
\end{align*}
From the construction above, we know that $(\tilde{Z}_k)_{k\ge1}$ is a symmetric random walk with iid steps with values $\pm |e^{tA}(x-y)|$. Using the central limit theorem we find for sufficiently large values of $k\geq k_0$ and some constant $C=C(k_0)$
\begin{equation}\label{lll3}\begin{aligned}
    \Pp\left(T^S>\frac12\,\gamma(\delta)k\right)
    &= 2\,\Pp\left(0\leq \frac{Z_k}{|e^{tA}(x-y)|\sqrt{\big[\frac{\gamma(\delta)k}{2}\big]}} \leq {\left[\frac{\gamma(\delta)k}{2}\right]}^{-1/2}\right)\\
    &\leq \frac{C}{\sqrt{2\pi}} \int_{0}^{  {\left[\frac{\gamma(\delta)k}{2}\right]}^{-1/2}} e^{-u^2/2}\,du\\
    &\leq \frac{C_{\gamma(\delta)}}{\sqrt k}.
\end{aligned}\end{equation}

Combining \eqref{lll1}, \eqref{lll2} and \eqref{lll3} gives for all $x,y\in\R^n$ with $|x-y| \le\delta(C_A\|\bar{B}\|)^{-1}$, $t\ge t_1+\cdots+t_k$ and $k\geq k_0$ that
$$
    \Pp\big(T^S>k\big)
    \le \frac{C_{\gamma(\delta)}}{\sqrt k}+\frac{4(1-\gamma(\delta))}{\gamma(\delta)k}.
$$
Finally, \eqref{proofs6} yields for all $x,y\in\R^n$ with $|x-y|\le \delta(C_A\|\bar{B}\|)^{-1}$, $t\ge t_1+\cdots+t_k$ and $k\ge1$, that
\begin{equation}\label{proofs7}
  \|\delta_{e^{tA}(x-y)}*\mu_{t_1,\cdots,t_k}-\mu_{t_1,\cdots,t_k}\|_{\var}
    \le \frac{C_{1,\delta,\nnu}}{\sqrt k}.
\end{equation}

\medskip
\noindent
\emph{Step 6}.
If we combine \eqref{proofs2}, \eqref{proofs3} and \eqref{proofs7}, we obtain that for all $x, y\in\R^n$ with $|x-y|\le \delta(C_A\|\bar{B}\|)^{-1}$,
 \begin{equation}\begin{aligned}\label{proofs8}
 \|&P_t(x,\cdot)-P_t(y,\cdot)\|_{\var}\\
 &\le 2e^{-C_\varepsilon t}
  +C_{1,\delta,\nnu}\sum_{k=1}^\infty\frac{1}{\sqrt{k}} \mathop{\int\cdots\int}\limits_{t_1+\cdots+t_k\le t<t_1+\cdots+t_{k+1}} \!\!\!\!\!\!\!\!C_\varepsilon^{k+1}e^{-C_\varepsilon(t_1+\cdots+t_{k+1})}\,dt_1\cdots dt_{k+1}\\
    &\le 2e^{-C_\varepsilon t}+C_{1,\delta,\nnu}e^{-C_\varepsilon t}\sum_{k=1}^\infty \frac{C_\varepsilon^{k+1}}{\sqrt{k}} \mathop{\int\cdots\int}\limits_{t_1+\cdots+t_k\le t}\!\!dt_1\cdots dt_{k}\\
    &\le 2e^{-C_\varepsilon t}+C_{1,\delta,\nnu}C_\varepsilon \sum_{k=1}^\infty\frac{C_\varepsilon^k t^k}{\sqrt{k}\,k!}e^{-C_\varepsilon t}\\
    &\le 2e^{-C_\varepsilon t}+\frac{\sqrt{2}C_{1,\delta,\nnu} C_\varepsilon (1-e^{-C_\varepsilon t})}{\sqrt{C_\varepsilon t}}\\
    &\le \frac{C_{2,\epsilon,\delta,\nnu}}{\sqrt{t}},
\end{aligned}\end{equation}
where the penultimate inequality follows as in \cite[Proposition 2.2]{SW}.

For any $x,$ $y\in\R^n$, set $k=\left[\frac{C_A\|\bar{B}\||x-y|}{\delta}\right]+1$. Pick $x_0, x_1, \ldots, x_k\in\R^n$ such that $x_0=x$, $x_k=y$ and $|x_i-x_{i-1}|\le \delta(C_A\|\bar{B}\|)^{-1}$ for $1\le i\le k$. By \eqref{proofs8},
\begin{align*}
    \|P_t(x,\cdot)-P_t(y,\cdot)\|_{\var}
    &\le \sum_{i=1}^k\|P_t(x_i,\cdot)-P_t(x_{i-1},\cdot)\|_{\var}\\
    &\le \frac{C_{\epsilon,\delta,\nnu,A,B}(1+|x-y|)}{\sqrt{t}},
\end{align*}
which finishes the proof of \eqref{th21}.
\end{proof}

\medskip

The following two lemmas have been used in the proof of Theorem \ref{th1} above. For the sake of completeness we include their proofs.

\begin{lemma}\label{lemmaproofth12}
    Let $B\in\R^{n\times d}$ and $(Z_t)_{t\ge0}$ be a $d$-dimensional L\'{e}vy process with characteristic exponent $\Phi$ as in \eqref{ou2}. Then, $(Z^B_t)_{t\ge0}:=(BZ_t)_{t\ge0}$ is a L\'{e}vy process on (a subspace of) $\R^n$, and the corresponding characteristic exponent is
    $$
        \R^n\ni\xi \mapsto \Phi_B(\xi):=\Phi(B^\top\xi).
    $$
    The L\'evy triplet $(Q_B,b_B,\nu_B)$ of $(Z^B_t)_{t\ge0}$ is given by $ Q_B=BQB^\top$, $\nu_B(C) = \nu\{y : By\in C\}$ and
    $$
        b_B=Bb+\int_{x\neq 0} Bx\,\big(\I_{\{z\in\R^d:|z|\le 1\}}(Bx) - \I_{\{z\in\R^d:|z|\le 1\}}(x)\big)\, \nu(dx).
    $$
\end{lemma}

\begin{proof}
    For all $\xi\in\R^n$ and $t\ge0$, we have
    $$
        \Ee(e^{i\scalp{\xi}{Z^B_t}})
        =\Ee(e^{i\scalp{\xi}{BZ_t}})
        =\Ee(e^{i\scalp{B^\top\xi}{Z_t}})
        =e^{-t\Phi(B^\top\xi)}.
    $$
    The assertion follows from \eqref{ou2} and some straightforward calculations.
\end{proof}

\begin{lemma}\label{lemmaproofth11}
    Let $A\in\R^{n\times n}$, $B\in\R^{n\times d}$ and $(Z_t)_{t\ge0}$ be a $d$-dimensional L\'{e}vy process with the characteristic exponent $\Phi$ as in \eqref{ou2}. For all $t>0$ the random variables $\int_0^t e^{(t-s)A}B\,dZ_s$ and $\int_0^t e^{sA}B\,dZ_s$ have the same probability distribution. Furthermore, both random variables are infinitely divisible, and the characteristic exponent (log-characteristic function) is given by
    $$
        \R^n\ni\xi \mapsto \Phi_t(\xi):=\int_0^t\Phi\big(B^\top e^{sA^\top}\xi\big)\,ds.
    $$
\end{lemma}

\begin{proof}
    We first assume that $n=d$ and $B=\id_{\R^d}$. For any $t>0$, we can use Lemma \ref{lemmaproofth12} and follow the proof of \cite[(17.3)]{SA} to deduce
    $$
        \Ee\left[\exp\left(i\lrscalp{\xi}{\int_0^t e^{(t-s)A}\,dZ_s}\right)\right]
        =\exp\left[-\int_0^t\Phi(e^{(t-s)A^\top}\xi)\,ds\right]
    $$
    for all $\xi\in\R^d$. Similarly, for every $\xi\in\R^d$,
    $$
        \Ee\left[\exp\left(i\lrscalp{\xi}{\int_0^t e^{sA}\,dZ_s}\right)\right]
        =\exp\left[-\int_0^t\Phi(e^{sA^\top}\xi)\,ds\right].
    $$
    Since
    $$
        \exp\left[-\int_0^t\Phi(e^{(t-s)A^\top}\xi)\,ds\right]
        =\exp\left[-\int_0^t\Phi(e^{sA^\top}\xi)\,ds\right],
    $$
    it follows that $\int_0^t e^{(t-s)A}B\,dZ_s$ and $\int_0^t e^{sA}B\,dZ_s$ have the same law.

    Now replace in the preceding calculations $A$ with $\frac 1k\,A$, $k\geq 1$, and set $Y_k := \int_0^t e^{s\,\frac 1k\, A}\,dZ_s$. Denote by $Y_k^{(j)}$, $1\leq j\leq k$, independent copies of $Y_k$. It is straightforward to see that $\sum_{j=1}^k Y_k^{(j)}$ and $\int_0^t e^{sA}\,dZ_s$ have the same law. This proves the infinite divisibility.

    If $n\neq d$, we consider, as in Lemma \ref{lemmaproofth12}, the L\'{e}vy process $(Z^B_t)_{t\ge0}:=(BZ_t)_{t\ge0}$ on (a subspace of) $\R^n$. Then, for any $\xi\in\R^n$,
    $$
        \Ee\left[\exp\left(i\lrscalp{\xi}{\int_0^t e^{(t-s)A}B\,dZ_s}\right)\right]
        =\Ee\left[\exp\left(i\lrscalp{\xi}{\int_0^t
        e^{(t-s)A}\,dZ^B_s}\right)\right],
    $$
    and the claim follows from the first part of our proof.
\end{proof}

The following result presents the upper estimate for the distribution of the maximum of a symmetric random walk, by using the reflection principle. Since we could not find a precise reference in the literature, we include the complete proof for the readers' convenience.

\begin{lemma}\label{lemmaproofth13}
    Consider a random walk $(S_i)_{i\ge1}$ on $\Z$ with iid steps, which attain the values $-1$, $1$ and $0$ with probabilities $(1-r)/2$, $(1-r)/2$ and $r$ $(0\le r<1)$, respectively. Then for any positive integers $a$ and $k$, we have
\begin{equation}\label{lemmaproofth1311}
    2\Pp(S_k> a)
    \le \Pp\Big(\max_{i\le k}S_i\ge a\Big)
    \le 2\Pp(S_k\ge a)
\end{equation}
    and
$$
    2\Pp\big(0<S_k<a\big)
    \le \Pp\Big(\max_{1\le i\le k}S_i < a\Big)\le 2\Pp\big(0\le S_k\le a\big).
$$
\end{lemma}


\begin{proof}[Proof]
    Fix any positive integer $a$ and define $\tau :=\tau_a := \inf\{i\ge 1\::\: S_i=a\}$. Since the random walk has iid steps, it is obvious that $(S_{i+\tau}-S_\tau)_{i\geq 0}$ and $(S_i)_{i\geq 0}$ are independent random walks having the same law. Observing that $S_\tau = a$ and $\left\{\max_{i\le k}S_i\ge a\right\} = \{\tau \leq k\}$ we find, therefore,
    \begin{align*}
    \Pp\Bigl(\max_{i\le k}S_i\ge a\Bigr)
    &=\Pp\Bigl(\max_{i\le k}S_i\ge a,\; S_k\ge a\Bigr) + \Pp\Bigl(\max_{i\le k}S_i\ge a,\; S_k < a\Bigr)\\
    &=\Pp(S_k\ge a)+\Pp(\tau\leq k,S_k<S_\tau)\\
    &=\Pp(S_k\ge a)+\Pp(\tau\leq k,S_k>S_\tau)\\
    &=\Pp(S_k\ge a)+\Pp(S_k>a).
    \end{align*}
    From this we conclude that
    $$
        2\Pp(S_k\ge a)
        \ge \Pp\Bigl(\max_{i\le k}S_i\ge a\Bigr)
        \ge 2\Pp(S_k>a).
    $$
    Since $\Pp(S_k\ge0)=\Pp(S_k\le0)\ge1/2$, we see
    \begin{align*}
        \Pp\Bigl(\max_{i\le k}S_i< a\Bigr)
        &=1-\Pp\Bigl(\max_{i\le k}S_i\ge a\Bigr)\\
        &\le 1-2\Pp(S_k>a)\\
        &\le 2\big(\Pp(S_k\ge0)-\Pp(S_k>a)\big)\\
        &=2\Pp(0\le S_k\le a);
    \end{align*}
    the other inequality follows similarly if we use $\Pp(S_k>0)=\Pp(S_k<0)\le 1/2$.
\end{proof}

\bigskip

Next, we turn to the proof of Theorem \ref{coup}.

\begin{proof}[Proof of Theorem \ref{coup}]
    \emph{Step 1}. As in the proof of 
     Lemma \ref{lemmaproofth11} we may, without loss of generality, assume that $n=d$ and $B=\id_{\R^d}$. For $t>0$, denote by $\mu_t$ the law of $X_t^0:=\int_0^t e^{(t-s)A}\,dZ_s$. According to Lemma \ref{lemmaproofth11}, the law $\mu_t$ is an infinitely divisible probability distribution, and the characteristic exponent of $\mu_t$ is given by
    $$
        \Phi_t(\xi):=\int_0^t\Phi\big(e^{sA^\top}\xi\big)\,ds.
    $$
    Since the driving L\'evy process $(Z_t)_{t\ge0}$ has no Gaussian part, the L\'evy {triplet} $(0,b_t,\nu_t)$ of $\Phi_t$ is given by, cf.\ \cite[Theorem 3.1]{SAT},
    \begin{gather*}
        \nu_t(C)= \int_0^t\nu(e^{-sA}C)\,ds,\qquad C\in\mathscr{B}(\R^d\setminus\{0\}),\\
        b_t= \int_0^t e^{sA}b\,ds + \int_{z\neq 0} \int_0^t e^{sA}z\Big(\I_{\{|z|\le 1\}} \big(e^{sA}z\big) - \I_{\{|z|\le1\}}(z) \Big)\,ds\,\nu(dz).
    \end{gather*}
    For every $r>0$, let $\{\mu_t^r, t\ge0\}$ be the family of infinitely divisible probability measures on $\R^d$ whose Fourier transform is of the form $\widehat{\mu}^r_t(\xi)=\exp(-\Phi_{t,r}(\xi))$, where
    $$
        \Phi_{t,r}(\xi) = \int_{|z|\le r} \left(1-e^{i\scalp{\xi}{z}}+i\scalp{\xi}{z}\right)\,\nu_t(dz)
    $$
    with $\nu_t$ as above.

    Set $h(t):=1\big/\varphi_t^{-1}(1)$. Following the proof of \cite[Propostion 2.2]{SSW}, the conditions \eqref{coup1} and \eqref{coup2} ensure that there exists $t_1>0$ such that for all $t\ge t_1$, the measure $\mu^{h(t)}_t$ has a density $p^{h(t)}_t\in C^{n+2}_b(\R^d)$; moreover,
    \begin{equation}\label{proofcoup1}
       \big|\nabla p^{h(t)}_t(y)\big|
       \leq c(n,\Phi)\,h(t)^{-(n+1)}\big(1+h(t)^{-1}|y|\big)^{-(n+1)}
    \end{equation}
    holds for all $y\in\R^d$.

\medskip
\noindent
\emph{Step 2}.
    For $r>0$ and $\xi\in\R^d$, define
    $$
    \Psi_{t,r}(\xi)
    := \Phi_t(\xi)-\Phi_{t,r}(\xi)
    = \int_{|z|>r}\left(1-e^{i\scalp{\xi}{z}}\right)\,\nu_t(dz) - i\lrscalp{\xi}{\int_{1<|z|\le r}z\,\nu_t(dz)-b_t}.
    $$
    Since $\Psi_{t,r}$ is given by a L\'evy-Khintchine formula, it is the characteristic exponent of some $d$-dimensional infinitely divisible random variable. Let $\{\mmu_t^{r}, t\ge 0\}$ be the family of infinitely divisible measures whose Fourier transforms are of the form
    $\widehat{\mmu}^{r}_t(\xi)=\exp(-\Psi_{t,r}(\xi))$.
    Clearly, $\mu_t=\mu_t^r*\mmu_t^{r}$ for all $t,r>0$.

    Let $P_t(x,\cdot)$ and $P_t$ be the transition function and the transition semigroup of the Ornstein-Uhlenbeck process $\{X_t^x\}_{t\ge0}$ given by \eqref{ou1}. For all $f\in B_b(\R^d)$ we have
    \begin{align*}
        P_tf(x)
        &=\int f\bigl(e^{tA}x+z\bigr)\,\mu_t(dz)\\
        &=\int f\bigl(e^{tA}x+z\bigr)\,\mu^r_t*\mmu_t^{r}(dz)\\
        &=\iint f\bigl(e^{tA}x+z_1+z_2\bigr)\,\mmu_t^{r}(dz_1)\,\mu^r_t(dz_2).
    \end{align*}

    Taking $r=h(t)$ we get, using the conclusions of step 1, that for all $t\ge t_1$ and $x\in\R^d$,
    \begin{align*}
        P_tf(x)
        &= \int p^{h(t)}_t(z_2)\,dz_2 \int f\bigl(e^{tA}x+z_1+z_2\bigr)\,\mmu_t^{h(t)}(dz_1)\\
        &= \int p^{h(t)}_t\bigl(z_2-e^{tA}x\bigr)\,dz_2 \int f(z_1+z_2)\,\mmu_t^{h(t)}(dz_1).
    \end{align*}
    If $\|f\|_\infty\le 1$, then
    $$
        \bigg\|\int f(z_1+\cdot)\,\mmu_t^{h(t)}(dz_1)\bigg\|_\infty
        \le \|f\|_\infty\, \mmu_t^{h(t)}(\R^d) \le 1.
    $$

\medskip
\noindent
\emph{Step 3}.
    For all $x,y\in\R^d$,
    \begin{equation}\label{proofcoup2}\begin{aligned}
    \|P_t(x,&\cdot)-P_t(y,\cdot)\|_{\var}\\
    &= \sup_{\|f\|_\infty\le 1}\big|P_tf(x)-P_tf(y)\big|\\
    &= \sup_{\|f\|_\infty\le 1} \bigg|\int p^{h(t)}_t\bigl(z_2-e^{tA}x\bigr)\,dz_2 \int f(z_1+z_2)\,\mmu_t^{h(t)}(dz_1)\\
    &\qquad\qquad\quad\mbox{}-\int p^{h(t)}_t\bigl(z_2-e^{tA}y\bigr)\,dz_2 \int f(z_1+z_2)\,\mmu_t^{h(t)}(dz_1)\bigg|\\
    &\le \sup_{\|g\|_\infty\le 1} \bigg|\int g(z)p^{h(t)}_t\bigl(z-e^{tA}x\bigr)\,dz
    - \int g(z)p^{h(t)}_t\bigl(z-e^{tA}y\bigr)\,dz\bigg|\\
    &= \sup_{\|g\|_\infty\le 1} \bigg|\int g(z)\Big(p^{h(t)}_t\bigl(z-e^{tA}x\bigr)-p^{h(t)}_t\bigl(z-e^{tA}y\bigr)\Big)\, dz\bigg|\\
    &= \int \Big|p^{h(t)}_t\bigl(z-e^{tA}x\bigr)-p^{h(t)}_t\bigl(z-e^{tA}y\bigr)\Big|\,dz.
    \end{aligned}\end{equation}

    With the argument used in the proof of \cite[Theorem 3.1]{SSW}, \eqref{coup3} follows from \eqref{proofcoup1} and \eqref{proofcoup2}.

  \medskip
\noindent
\emph{Step 4}.
    By assumption \eqref{coup4}, $$\varphi_\infty(\rho):=\sup_{|\xi|\le \rho}\int_0^\infty\Re\Phi\big(B^\top e^{sA^\top}\xi\big)\,ds$$ is finite on $(0,\infty)$; in particular,
    $\varphi_\infty^{-1}(1)\in(0,\infty]$.
    On the other hand, for any $t\ge t_0$, according to \eqref{coup1},
       \begin{align*}\int \exp\left(-\int_0^t\Re\Phi\big(B^\top e^{sA^\top}\xi\big)\,ds\right)& |\xi|^{n+2} \,d\xi\\
       &\le\int \exp\left(-\int_0^{t_0}\Re\Phi\big(B^\top e^{sA^\top}\xi\big)\,ds\right) |\xi|^{n+2} \,d\xi\\
       &=:C(t_0)<\infty. \end{align*} Since the function $t\mapsto \varphi^{-1}_t(1)$ is decreasing on $(0,\infty]$, \eqref{coup2} holds.
    This finishes the proof.
    \end{proof}

\section{Appendix}\label{sec-appendix}
\subsection{Gradient Estimates for Ornstein-Uhlenbeck Processes}\label{subsec-appendix1}

Motivated by \cite[Theorem 1.3]{SSW}, we have the following results
for gradient estimates of an Ornstein-Uhlenbeck process. This is the
counterpart of Theorem \ref{coup}. For $t, \rho>0$, define
$$
    \varphi(\rho)
    :=\sup_{|\xi|\le \rho}\Re\Phi\bigl(B^\top \xi\bigr)\quad\textrm{ and }\quad \varphi_t(\rho)
    :=\sup_{|\xi|\le \rho}\int_0^t\Re\Phi\big(B^\top e^{sA^\top}\xi\big)\,ds,
$$
where $\Phi$ is the characteristic exponent of the driving L\'evy
process $(Z_t)_{t\geq 0}$ from \eqref{ou1}.

\begin{theorem}\label{strong}
    Let $P_t(x,\cdot)$ be the transition function of the $n$-dimensional Ornstein-Uhlenbeck process $\{X_t^x\}_{t\ge0}$ given by \eqref{ou1}. Assume that
    \begin{equation}\label{strong1}
        \liminf_{|\xi|\rightarrow\infty} \frac{\Re\Phi\bigl(B^\top \xi\bigr)}{\log (1+|\xi|)} =\infty.
    \end{equation}
 If for any $C>0$,
    \begin{equation}\label{strong2}
        \int \exp\left[-C t\Re\Phi\bigl(B^\top\xi\bigr)\right] |\xi|^{n+2} \,d\xi
        = \mathsf{O}\left(\varphi^{-1}\Big(\frac{1}{t}\Big)^{2n+2}\right) \qquad\text{as\ \ } t\to 0,
    \end{equation}
    then there exists $c>0$ such that for all $t>0$ and $f\in{B}_b(\R^n)$,
    \begin{equation}\label{strong3}
        \|\nabla P_t f\|_\infty\le c\|f\|_\infty \,\varphi^{-1}\Big(\frac{1}{t\wedge1}\Big).
    \end{equation}
    If, in addition, $$\xi\mapsto\int_0^\infty\Re\Phi\big(B^\top e^{sA^\top}\xi\big)\,ds\quad \textrm{ is locally bounded},$$
       then there exist $t_1,c>0$ such that for $t\ge t_1$ and
       $f\in{B}_b(\R^n)$,
    \begin{equation}\label{strong4}
        \|\nabla P_t f\|_\infty
        \leq c\,\|f\|_\infty \bigg[\|e^{tA}\|\,\varphi_t^{-1}(1)\bigg],
    \end{equation}
    where $\|M\|=\sup_{|x|\le1}|Mx|$ denotes the norm of the matrix of $M$.
\end{theorem}

To illustrate the power of Theorem \ref{strong}, we consider
\begin{example}\label{examplegradient}
      Let $\mu$ be a finite nonnegative measure on the unit sphere $\sfera\subset \R^n$ and assume that $\mu$ is nondegenerate in the sense that its support is not contained in any proper linear subspace of $\R^n$. Let $\alpha\in(0,2)$, $\beta\in (0,\infty]$ and assume that the L\'evy measure $\nu$ satisfies
$$
    \nu(C)
    \geq \int_0^{r_0}\int_{\sfera}\I_C(s\theta)s^{-1-\alpha} \,ds\, \mu(d\theta)
    + \int_{r_0}^\infty\int_{\sfera}\I_C(s\theta)s^{-1-\beta}\,ds\, \mu(d\theta)
$$
    for some constant $r_0>0$ and all $C\in \Bb(\R^n\setminus\{0\})$. Consider the following Ornstein-Uhlenbeck process $X_t$ on $\R^n$ given by
    $$
        dX_t=AX_t\,dt+dZ_t,
    $$
    where $(Z_t)_{t\ge0}$ is a L\'evy process on $\R^n$ with the L\'{e}vy measure $\nu$. By Theorem \ref{strong} there exists a constant $c>0$ such that for all $t>0$ and $f\in B_b(\R^n)$,
    $$
        \|\nabla P_tf\|_\infty
        \le  c\,\|f\|_\infty \, (t\wedge 1)^{-1/\alpha}.
    $$
    Furthermore, if the real parts of all eigenvalues of $A$ are negative, then there exists a constant $c>0$ such that for all $t>0$ and $f\in B_b(\R^n)$,
    $$
        \|\nabla P_tf\|_\infty
        \le  c\,\|f\|_\infty \,\frac{\|e^{tA}\|\quad}{(t\wedge 1)^{1/\alpha}}.
    $$
    \end{example}

\medskip

Recently, F.-Y.\ Wang \cite[Theorem 1.1]{W2}
has presented explicit gradient estimates for Ornstein-Uhlenbeck processes, by assuming
that the corresponding L\'{e}vy measure has absolutely continuous
(\emph{with respect to Lebesgue measure}) lower bounds. Since lower
bounds of L\'{e}vy measure in Example \ref{examplegradient} could be much
irregular, Theorem \ref{strong} is more applicable than
\cite[Theorem 1.1]{W2}.

\medskip

\begin{proof}[Sketch of the Proof of Theorem \ref{strong}]
    Assuming the conditions \eqref{strong1} and \eqref{strong2}, we can mimic the proof of \cite[Theorem 3.2]{SSW} to show that there exist $t_1, C>0$ such that for all $x,y\in\R^n$ and $t\le t_1$,
\begin{equation*}\label{proofstrong}
    \|P_t(x,\cdot)-P_t(y,\cdot)\|_{\var}
    \le C\,|e^{tA}(x-y)|\,\varphi^{-1}\Big(\frac{1}{t}\Big).
\end{equation*}
    Thus we can apply to find for all $f\in B_b(\R^n)$ with $\|f\|_\infty= 1$,
\begin{equation}\label{proofstrong}\begin{aligned}
    |\nabla P_tf(x)|
    &\leq \limsup_{y\to x}\frac{|P_tf(x)-P_tf(y)|}{|y-x|}\\
    &\leq \limsup_{y\to x}\frac{\sup_{\|w\|_\infty\le 1}|P_tw(x)-P_tw(y)|}{|y-x|}\\
    &\leq \limsup_{y\to x}\frac{\|P_t(x,\cdot)-P_t(y,\cdot)\|_{\var}}{|y-x|}\\
    &\leq C\,\|e^{tA}\|\,\varphi^{-1}\Big(\frac{1}{t}\Big)\\
    &\leq \Big[C\,\sup_{s\le t_1}\|e^{sA}\|\Big]\,\varphi^{-1}\Big(\frac{1}{t}\Big).
\end{aligned}\end{equation}  Because of the Markov property of the semigroup $P_t$, the function $$t\mapsto \sup_{f\in B_b(\R^n),\, \|f\|_\infty =1}{ \|\nabla P_t f\|_\infty}$$ is deceasing. Combining this and \eqref{proofstrong} yields \eqref{strong3}.

    The assertion \eqref{strong4} follows if we combine the above argument with \eqref{coup3}: there exist $t_2,C>0$ such that for all $x,y\in\R^n$ and $t\ge t_2$,
\begin{gather*}
    \|P_t(x,\cdot)-P_t(y,\cdot)\|_{\var}\leq C\,|e^{tA}(x-y)|\,\varphi_t^{-1}(1).
\qedhere
\end{gather*}
 \end{proof}

\subsection{Proof of Proposition \ref{improvement}}\label{subsec-appendix2}
\begin{proof}[Proof of Proposition \ref{improvement}]
    Because of \eqref{wang22}, we can choose a closed subset $F\subset \overline{B(z_0, \varepsilon)}$ such that $0\notin F$ and
    $$
        \int_F \frac{dz}{\rho_0(z)} < \infty.
    $$
    By the Cauchy-Schwarz inequality, we have
    $$
        \left(\int_F \rho_0(z)\,dz\right)^{-1}
        \le \frac{1}{\Leb(F)^2} \int_F \frac{dz}{\rho_0(z)} < \infty.
    $$
    Hence,
    $$
        K:=\int_F\rho_0(z)\,dz>0.
    $$

    Since $F$ is a compact set and $0\notin F$, there exists some $\delta_0>0$ such that
$
        0\notin F+\overline{B(0,\delta_0)},
    $ where $F+\overline{B(0,\delta_0)}:=\{a+b:a\in F, |b|\le \delta_0\}.$
Since $\rho_0$ is locally integrable, we know that $$
        K\le\int_{F+\overline{B(0,\delta_0)}}\rho_0(z)\,dz<\infty.
    $$
    The remainder of the proof is now similar to the argument which shows that the shift $x\mapsto \|f(\cdot - x)-f\|_{L^1}$, $f\in L^1(\R^d,\Leb)$, is continuous, see e.g.\ \cite[Lemma 6.3.5]{STR} or \cite[Theorem 14.8]{RSCC}: choose $\chi\in C_c^\infty (\R^d)$ such that $\supp\chi\subset F+\overline{B(0,\delta_0)}$ and $$
        \int_{F+\overline{B(0,\delta_0)}}|\rho_0(z)-\chi(z)|\, dz\le \frac{K}{4}.
    $$
    Therefore, for any $x\in\R^d$ with $|x|\le \delta_0$, we obtain
    \begin{align*}
    \int_F & |\rho_0(z)-\rho_0(z-x)|\,dz\\
    &\le \int_F|\rho_0(z)-\chi(z)|\,dz
        +\int_F|\chi(z)-\chi(z-x)|\,dz
        +\int_F|\rho_0(z-x)-\chi(z-x)|\,dz\\
    &= \int_F|\rho_0(z)-\chi(z)|\,dz
        +\int_F|\chi(z)-\chi(z-x)|\,dz
        +\int_{F+x}|\rho_0(z)-\chi(z)|\,dz\\
    &\le 2\int_{F+\overline{B(0,\delta_0)}}|\rho_0(z)-\chi(z)|\,dz
        +\int_F|\chi(z)-\chi(z-x)|\,dz\\
    &\le \frac{K}{2}+\int_F|\chi(z)-\chi(z-x)|\,dz.
    \end{align*}
    By the dominated convergence theorem we see that
    $$
        x\mapsto\int_F|\chi(z)-\chi(z-x)|\,dz
    $$
    is continuous on $\R^d$. Therefore, there exists $0<\delta\le\delta_0$ such that
    $$
        \sup_{x\in\R^d, |x|\le \delta}\int_F|\chi(z)-\chi(z-x)|\,dz\le \frac{K}{4}
    $$
    and, in particular,
    $$
        \sup_{x\in\R^d, |x|\le \delta}\int_F|\rho_0(z)-\rho_0(z-x)|\,dz\le \frac{3K}{4}.
    $$
    Using $2(a\wedge b)=a+b-|a-b|$ for all $a,b\ge0$, we get
    \begin{align*}
        \inf_{x\in\R^d, |x|\le \delta} & \int_F\big(\rho_0(z)\wedge\rho_0(z-x)\big)\,dz\\
        &=\frac{1}{2}\inf_{x\in\R^d, |x|\le \delta} \bigg[\int_F\big(\rho_0(z)+\rho_0(z-x)\big)\,dz
            - \int_F\big|\rho_0(z)-\rho_0(z-x)\big|\,dz\bigg]\\
        &\ge \frac{1}{2}\int_F\rho_0(z)\,dz
            - \frac{1}{2}\sup_{x\in\R^d, |x|\le \delta}\int_F\big|\rho_0(z)-\rho_0(z-x)\big|\,dz\\
        &\ge \frac{K}{8}>0.
    \end{align*}
    This finishes the proof.
\end{proof}

\begin{ack}
Financial support through DFG (grant Schi 419/5-1) and DAAD (PPP Kroatien) (for Ren\'{e} L.\ Schilling)
and the Alexander-von-Humboldt Foundation
  and the Natural Science Foundation of Fujian $($No.\ 2010J05002$)$
(for Jian Wang) is gratefully acknowledged.
\end{ack}


\begin{thebibliography}{99}
\bibitem{AB}
Aliprantis, C.D.\ and Burkinshaw, O.: \emph{Principles of Real Analysis (3rd ed.)}, Academic Press, San Diego, California
1998,

\bibitem{Bern}
Bernstein, D.S.: \emph{Matrix Mathematics: Theory, Facts, and
Formulas with Application to Linear Systems Theory}, Princeton
University Press, Princeton 2005.


\bibitem{CG}
Cranston,  M.\ and Greven, A.: Coupling and harmonic functions in
the case of continuous time Markov processes, \emph{Stoch.\
Proc.\ Appl.} \textbf{60} (1995), 261--286.

\bibitem{CW}
Cranston,  M.\ and Wang, F.-Y.: A condition for the equivalence of
coupling and shift-coupling, \emph{Ann.\ Probab.} \textbf{28} (2000),
1666--1679.

\bibitem{Li}
Lindvall, T.: \emph{Lectures on the Coupling Method}, Wiley, New
York 1992.


\bibitem{NS}
Nourdin, I.\ and Simon, T.: On the absolute continuity of L\'{e}vy processes with drift, \emph{Ann.\ Probab.} \textbf{34} (2006), 1035--1051.

\bibitem{PZ}
Priola, E.\ and Zabczyk, J.: Liouville theorems for non-local operators, \emph{J.\ Funct.\ Anal.} \textbf{216} (2004), 455--490.

\bibitem{PZ2}
Priola, E.\ and Zabczyk, J.: Densities for Ornstein-Uhlenbeck
processes with jumps, \emph{Bull.\ London Math.\ Soc.} \textbf{41} (2009), 41--50.

\bibitem{SAT}
Sato, K.\ and Yamazato, M.: Operator-self-decomposable distributions
as limit distributions of processes of Ornstein-Uhlenbeck type,
\emph{Stoch.\ Proc.\ Appl.} \textbf{17} (1984), 73--100.


\bibitem{SA}
Sato, K.: \emph{L\'{e}vy processes and Infinitely Divisible
Distributions}, Cambridge University Press,
Cambridge 1999.

\bibitem{RSCC}
Schilling, R.L.: \emph{Measures, Integrals and Martingales}, Cambridge University Press, Cambridge 2005.

\bibitem{SSW}
Schilling, R.L., Sztonyk, P.\ and Wang, J.: Coupling property and gradient estimates of L\'evy processes via symbol, to
appear in \emph{Bernoulli}, 2011. See also arXiv 1011.1067

\bibitem{SW}
Schilling, R.L.\ and Wang, J.: On the coupling property of L\'evy
processes, to appear in \emph{Ann.\ Inst.\ Henri Poincar\'e:
Probab.\ Stat.}, 2010. See also arXiv 1006.5288

\bibitem{STR}
Stroock, D.W.: \emph{A Concise Introducation to the Theory of Integration (2nd ed.)}, Birkh\"{a}user, Boston 1994.

\bibitem{T}
Thorisson, H.: \emph{Coupling, Stationarity and Regeneration},
Springer, New York 2000.

\bibitem{wang1}
Wang, F.-Y.: Coupling for Ornstein-Uhlenbeck jump processes, to
appear in \emph{Bernoulli}, 2010. See also arXiv:1002.2890v5

\bibitem{W2}
Wang, F.-Y.:
Gradient estimate for Ornstein-Uhlenbeck jump processes, \emph{Stoch.\ Proc.\ Appl.} \textbf{121} (2011), 466--478.

\end{thebibliography}
\end{document}